\pdfoutput=1
\documentclass[12pt]{article}
\date{}
\usepackage{amsmath,mathrsfs,amsthm,amssymb}
\usepackage{geometry}
\usepackage{graphicx}
\usepackage{gensymb}
\usepackage{relsize,tikz-cd}
\usepackage{amsfonts}
\usepackage{commath}
\usepackage{subfig}
\title{Transverse and Legendrian invariants of cables in combinatorial link Floer homology}

\newtheorem{theorem}{Theorem}
\newtheorem{prop}{Proposition}[section]

\newtheorem{cor}{Corollary}[section]
\newtheorem{definition}{Definition}[section]
\newtheorem*{remark}{Remark}
\usepackage{url}
\makeatletter
\newcommand{\address}[1]{\gdef\@address{#1}}
\newcommand{\email}[1]{\gdef\@email{\url{#1}}}
\newcommand{\@endstuff}{\par\vspace{\baselineskip}\noindent\small
\begin{tabular}{@{}l}\scshape\@address\\\textit{E-mail address:} \@email\end{tabular}}
\AtEndDocument{\@endstuff}
\makeatother

\author{Apratim Chakraborty}
\address{Department of Mathematics, Stony Brook University, Stony Brook, NY 11794}
\email{apratim@math.stonybrook.edu}

\begin{document}    
\maketitle
\begin{abstract}
We study the Ozsv\'{a}th-Szab\'{o}-Thurston transverse invariant in combinatorial link Floer homology for certain transverse cables $\mathscr{L}_{p,q}$ of transverse link $L$ in $S^3$. Transverse cables $\mathscr{L}_{p,q}$ are constructed from the grid diagram of $L$. The main result is $\hat{\theta}(\mathscr{L}_{p,q})=0$ if and only if $\hat{\theta}(L)=0$ for $\frac{q}{p}$ sufficiently large. We also prove a similar result for invariants of Legendrian knots. Our proof uses an inclusion map $i$ of certain grid complexes associated to $L$ and $L_{p,q}$. We use these results to generate many infinite families of examples of Legendrian and transversely non-simple topological link types.      
\end{abstract}

\section{Introduction}
Transverse links play an important role in the study of contact structures in $3$-manifolds. However, it's often difficult to distinguish transverse links if they have the same topological link type and self-linking number. If any two transverse representatives in a topological link type with the same self-linking number are isotopic, then the topological link type is called transversely simple otherwise, it is called transversely non-simple. There are well-known examples of transversely simple link types.
 Eliashberg \cite{unknot} proved that unknot is transversely simple. Subsequently, Etnyre \cite{Etnyre} showed that torus knots are transversely simple. First examples of transversely non-simple link types were given by Birman and Menasco \cite{first}. Etnyre and Honda \cite{EtnyreHonda3} showed that $(2,3)$ cable of the $(2,3)$ torus knot was transversely non-simple and provided a classification in that link type. In the same vein, a link type is called Legendrian simple if any two representatives in that link type with same Thurston-Bennequin and rotation numbers are Legendrian isotopic. The classification problem of Legendrian links is a well-researched subject. The first example of Legendrian non-simple knot was given by Chekanov \cite{chek} using Chekanov-Eliashberg DGA which is one of the most powerful tools in the realm of Legendrian knot theory.  \\

The first effective transverse invariant was defined in knot Floer complex by Ozsv\'{a}th, Szab\'{o} and Thurston \cite{theta}. In this paper, we focus on the invariant $\hat{\theta}$ in the hat version of knot Floer homology, although stronger refinements can be obtained by using the properties of the filtered complex. Ng, Ozsv\'{a}th and Thurston \cite{ng} used $\hat{\theta}$ to reprove $(2,3)$ cable example and gave additional examples based on the refinement. Since then, $\hat{\theta}$ has been used quite a number of times to give examples of transversely non-simple link types. V\'{e}rtesi \cite{vv} proved a connected sum property of $\hat{\theta}$ and gave an infinite family of examples. Khandhawit and Ng \cite{kn} provided additional infinite families of examples by studying grid diagrams of certain families of $4$-braids. Baldwin \cite{Bald1} proved comultiplication property of $\hat{\theta}$ and found more infinite families. Those examples involved finding representatives $\mathcal{T}_{1}$ and $\mathcal{T}_{2}$ of a link type $\hat{\theta}(\mathcal{T}_{1}) =0$ and $\hat{\theta}(\mathcal{T}_{2}) \neq 0$. Lisca,  Ozsv\'{a}th, Stipsicz and Szab\'{o} \cite{LOSS} defined a more general version of $\theta$ (called LOSS invariant) for transverse links in contact $3$-manifolds. Ozsv\'{a}th and Stipsicz \cite{natural} showed transverse non-simplicity for a wide family of two-bridge knots by studying naturality properties of LOSS invariant. \\

In this paper, we study cables because they provide a natural avenue to look for examples of transversely non-simple link types, and also they provide interesting insight into some concordance invariants. We obtain grid diagrams for cables $L_{p,q}$ by subdividing the grid of $L$ where $q \geq p ( 2n(L)+ tb(L) )$ where $n(L)$ is the minimum grid number of a grid representing the Legendrian link $L$. The Legendrian and transverse the representatives of the cable corresponding to the constructed grid diagrams will be denoted by $\mathcal{L}_{p,q}$ and $\mathscr{L}_{p,q}$ respectively. In section 2, we will show that this bound on $q$ can be strengthened in individual cases by using some variables of the grid diagram. We prove that non-vanishing of $\hat{\theta}$ of a transverse link $L$ is equivalent to non-vanishing of $\hat{\theta}$ for some transverse representatives of the cable  $\mathscr{L}_{p,q}$.
 
\begin{theorem}\label{mostimp} 
For  $\frac{q}{p}$ sufficiently large, there is a transverse representative $\mathscr{L}_{p,q}$ of the $(p,q)$ cable of a transverse link $L$ such that, $\hat{\theta}( \mathscr{L}_{p,q})=0$ if and only if $\hat{\theta}(L)=0$.
\end{theorem}

This result shows that if non vanishing of $\hat{\theta}$ shows some link type is non-simple then its cables are also non-simple. So it provides new examples of infinitely many transversely non-simple topological link types by taking cables examples in \cite{vv,kn,Bald1}.\\

We can also prove a similar statement about Legendrian cables. We prove the following result about Legendrian invariants $\hat{{\lambda}}^{+}(L)$ and $\hat{{\lambda}}^{-}(L)$ \cite{theta} for a link $L$.

\begin{theorem}\label{chekanov} 
For $\frac{q}{p}$ sufficiently large, there is a Legendrian representative $\mathcal{L}_{p,q}$ of the $(p,q)$ cable of a Legendrian link $L$ such that, $\hat{{\lambda}}^{+}(\mathcal{L}_{p,q})= \hat{{\lambda}}^{-}(\mathcal{L}_{p,q})$ if and only if $\hat{{\lambda}}^{+}(L)=\hat{{\lambda}}^{-}(L)$.
\end{theorem}

As a corollary, we show that certain infinite family of cables of $m(5_{2})$ is Legendrian non-simple. These results about Legendrian cables complement work of LaFountain \cite{lafo1} and Tosun \cite{tosun} that showed Legendrian simplicity is preserved after cabling under certain conditions. Also, it will be interesting to see if all cables of non-simple Legendrian links are non-simple.  \\

The proofs of Theorem \ref{mostimp} and \ref{chekanov}  relies on a more general result about link Floer homology of cables. We consider a fully collapsed grid complex $\mathscr{C}_{{L}_{p,q}}$ [See Definition \ref{definec}] associated to the grid diagram of ${L}_{p,q}$, and we assign a chain complex $p\mathscr{C}$ (obtained by a change of variable from fully collapsed grid complex) to the grid diagram of $L$. Then we observe that there is a natural inclusion $i$ of this complex to the collapsed grid complex of the cable ${L}_{p,q}$. We prove that $i([x^+]) \neq 0$ if $[x^+] \neq 0$.

The paper is organized as follows. In Section 2, we review grid diagrams and give a prescription for generating Legendrian/transverse cables. In Section 3, we review grid homology and discuss the properties of a collapsed grid complex that we will be using. In Section 4, we prove the main theorems and give examples of Legendrian and transversely non-simple link types that can be obtained using those theorems.\\   
	
\date{\textbf{Acknowledgements:} I am grateful to my advisor, Olga Plamenevskaya, for suggesting the project and giving helpful advice throughout preparation. I would like to thank John Etnyre, John Baldwin, Lenny Ng and Jen Hom for helpful comments on the first draft.} 
\pagebreak
\section{Grid Diagrams}

Grid diagrams provide a combinatorial framework for studying Legendrian/transverse links, braids and knot Floer complexes (See \cite{Grid Homology for Knots and Links} for more details). A planar grid diagram $P$ with grid number $n$ is an $n \times n$ grid with squares marked with X’s and
O’s in a way that no square contains both X and O, and each row and each column contains exactly one X and one O. $\mathbb{X}$ will denote the set of squares marked with an X, and $\mathbb{O}$ the ones containing an O. Every planar grid diagram $P$ determines a diagram of an oriented link $L$ in the following way: In each row connect the O-marking to the X-marking, and in each column connect the X-marking to the O-marking with an oriented line segment, such that the vertical segments always pass over the horizontal ones.  We call $P$ a planar grid diagram for $L$. Conversely, every oriented link $L$ can be represented by some planar grid diagram. If $L$ is a link with $l$ components ($L_{1},L_{2},..$ and $L_{l}$), $\mathbb{X}_{i}$ (resp. $\mathbb{O}_{i}$)  denotes X marked squares (resp O marked squares)in $L_{i}$. Then, we can write $\mathbb{X}={\mathbb{X}}_{1} \cup {\mathbb{X}}_{2}\cup .. \cup {\mathbb{X}}_{l}$ and $\mathbb{O}={\mathbb{O}}_{1} \cup {\mathbb{O}}_{2}\cup .. \cup {\mathbb{O}}_{l}$.    \\

To work in Heegaard Floer homology setting, we find it convenient to transfer the diagram to torus $\mathbb{T}$. A toroidal grid diagram can be obtained by identifying the opposite sides of a planar grid diagram $P$: its top boundary segment with its bottom one and its left boundary segment with its right one. The resulting diagram $D$ in torus $\mathbb{T}$ is called a toroidal grid diagram, or simply a grid diagram of the link $L$.\\

\begin{figure}[!tbph]  
\begin{center}
\includegraphics[width=0.30\textwidth]{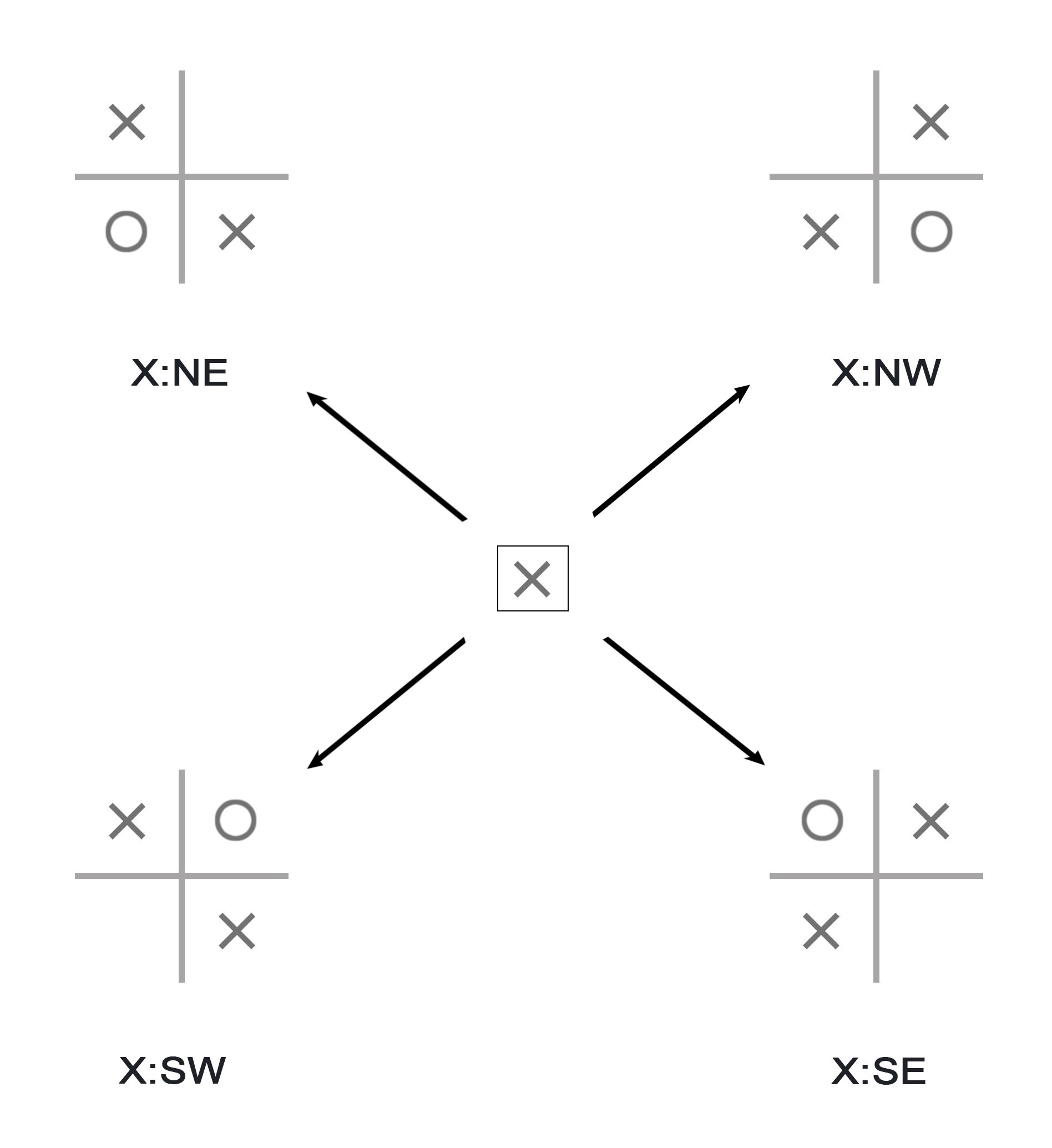}    

\caption{Stabilization Moves}\label{stab}
\end{center}
\end{figure}

There are certain moves of grid diagrams that are the equivalent to Reidemeister moves for knot diagrams. They are commutations of rows or columns and stabilizations. Let $D$ be an $n \times n$ grid diagram. We say that the $(n + 1) \times (n + 1)$ grid diagram $D'$ differs from $D$ by a stabilization (or that $D'$ is the stabilization of $D$), if it can be obtained from $D$ in the following way: choose a marked square in $D$, and erase the marking in it, in the other marked square in its row and in the other marked
square in its column. Then split the row and the column of the chosen marking in $D$ into two, that is, add a new horizontal and a new vertical line to get an ($n+1) \times (n+1)$ grid. There are four ways to insert markings in the two new rows and columns to have
a grid diagram. When the original square is marked with an X, these are called X:NE, X:NW, X:SE and X:SW [See Figure~\ref{stab}]. It turns out that it suffices to consider only these stabilizations for Reidemeister moves.

\subsection{Legendrian and transverse links}

Let us consider the standard tight contact structure $(\mathbb{R}^3,\xi_{st})$, with $\xi_{st}= ker (dz - y dx)$. An oriented link $L \subset \mathbb{R}^3$ is called $\mathbf{Legendrian}$ if it is everywhere tangent to $\xi_{st}$. An oriented link $L \subset \mathbb{R}^3$ is called $\mathbf{transverse}$ if it is everywhere transverse to $\xi_{st}$ and $dz - y dx > 0$ along the orientation. Any smooth link can be perturbed by a $C^0$ isotopy to be Legendrian or transverse. We say that two Legendrian (resp. transverse) links are $\mathbf{Legendrian}$ $\mathbf{isotopic}$ (resp. transversely isotopic) if they are isotopic through Legendrian links (resp. transverse links). We refer the reader to \cite{survey} for a thorough exposition of Legendrian and transverse links.\\

It's convenient to depict a Legendrian link is through its front projection or
projection in the $x-z$ plane. A generic front projection has three features: it has no
vertical tangencies; it is immersed except at cusp singularities; and at all crossings, the
strand of larger slope passes underneath the strand of the smaller slope.  Any front projection with these
features corresponds to a Legendrian link, with the $y$ coordinate given by the formula $y = \frac{dz}{dx}$.\\

The grid diagram $D$ can be viewed as the front projection of a Legendrian link via the following construction. First, we smooth north-west and south-east corners and turn south-west and north-east corners into cusps. Then, to avoid vertical tangencies, we tilt the diagram $45 \degree$ clockwise. Lastly, we reverse all the crossing to ensure the correct crossing convention for a Legendrian front projection [See Figure~\ref{leg}]. It is easy to see that if $D$ is a grid diagram of a link $L$, the Legendrian knot associated to $D$, denoted by $L_{D}$ is the Legendrian representative of $m(L)$. \\

It should be noted that there are couple of conventions for representing Legendrian/transverse links in grid diagram. Here, we adopt the convention of the book Grid Homology for Knots and Links by  Ozsv\'{a}th, Stipsicz and Szab\'{o} \cite{Grid Homology for Knots and Links}.
\begin{prop}\cite{Grid Homology for Knots and Links} Any Legendrian link type can be represented by some toroidal grid diagram. Two toroidal grid diagrams represent the same Legendrian link type if and only if they can be connected by a sequence of commutation and (de)stabilization of types X: NW and X: SE on the torus.
\end{prop}

\begin{figure}[!tbph]

\includegraphics[scale=0.13]{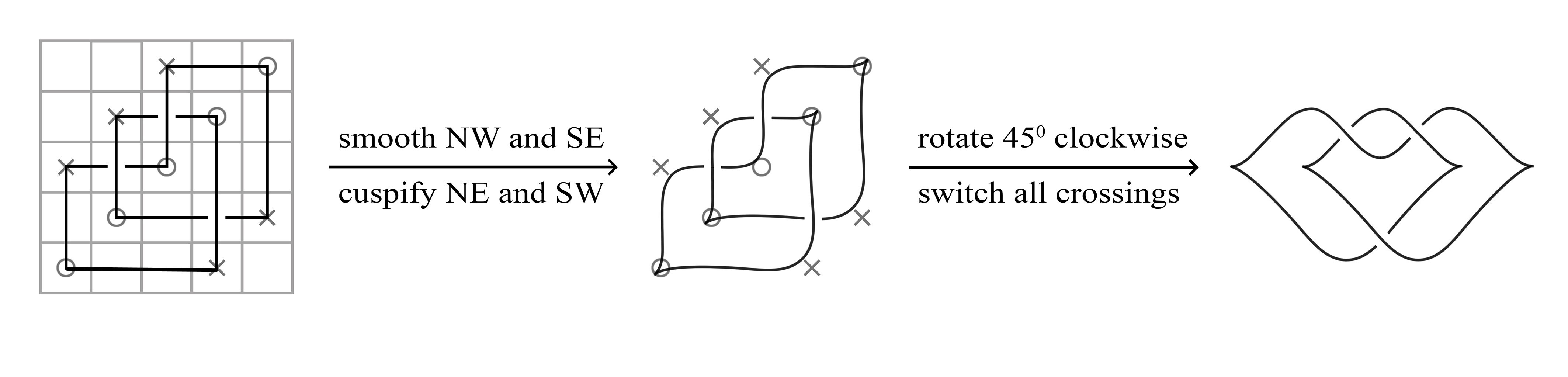}
\caption{ Getting the front projection of a Legendrian link from grid diagram}\label{leg}
\end{figure}

The two main classical invariants of a Legendrian link $L$ are are the $\mathbf{Thurston-Bennequin }  $ $\mathbf{invariant}$ $tb(L)$ and the $\mathbf{rotation}$ $\mathbf{number}$ $r(L)$. They can be easily defined in terms of the front projection $\textit{D}(L)$ of $L$. Let $wr(\textit{D}(L))$ denote the writhe of the projection. Then, \[ tb(L)= wr(\textit{D}(L)) - \frac{1}{2} \# \{ \text{cusps   in}  \textit{D}(L) \} \] and, 
\[r(L)=  \frac{1}{2} \# ( \{  \text{downward-oriented cusps} \} - \{ \text{upward-oriented cusps} \} ) \].

There are also formulas for these invariants  in terms of corners in the grid corresponding to the front projection which will be useful to us. Let $x_{NW}$ (and similarly $x_{SW}$,$x_{SE}$,$x_{NE}$) denote the number of north west (similarly south west, south-east, north-east) X markings, and 
define $o_{NW}$, $o_{SW}$, $o_{SE}$, $o_{SW}$ similarly. Then, 
\begin{equation}\label{eqn1}
r( L_{D}) = \frac{1}{2}(x_{NE}+o_{SW}-x_{SW}-o_{NE})
\end{equation}
and
\begin{equation}\label{eqn2}
tb( L_{D}) = - wr(D) - \frac{1}{2}(x_{NE}+o_{SW}+x_{SW}+o_{NE})
\end{equation}.

The transverse push-off $\mathcal{T}(L)$ of an oriented Legendrian link $L$ is the transverse link type which can be represented by transverse link arbitrarily close to $L$. Also, any transverse link can be represented as a transverse push-off of some Legendrian link. The main classical invariant of a transverse link $L$ is the \textbf{self-linking} \textbf{number} $sl(L)$. For a transverse push-off, $\mathcal{T}(L)$ , $sl(\mathcal{T}(L)) = tb(L) - r(L)$. We can also obtain a correspondence from transverse links to grid diagrams by thinking of the transverse link as a push-off of a Legendrian link $L$ and then taking the grid diagram $D_{L}$ corresponding to $L$.

\begin{prop}\cite{Grid Homology for Knots and Links} Any transverse link type can be represented by some toroidal grid diagram. Two toroidal grid diagrams represent the same transverse link type if and only if they can be connected by a sequence of commutation and (de)stabilization of types X: NW, X: SE and X: SW on the torus.
\end{prop}

\subsection{Grid diagrams for Cables} \label{constructed_grids}

The $(p,q)$- cable of a link $L$, denoted ${L}_{p,q}$ , is the satellite link with pattern
the $(p,q)$-torus knot $T_{p,q}$ (where $p$ indicates the longitudinal winding and $q$ indicates the meridional winding) and companion $L$. So, we can think of ${L}_{p,q}$ as the topological type of a link supported on the boundary of a tubular neighbourhood of $L$ with slope $\frac{p}{q}$ with respect to the standard framing of the torus, where the longitude is determined by the Seifert framing for $L$.  

\begin{figure}[!tbph]
\begin{center}
\includegraphics[scale=0.08]{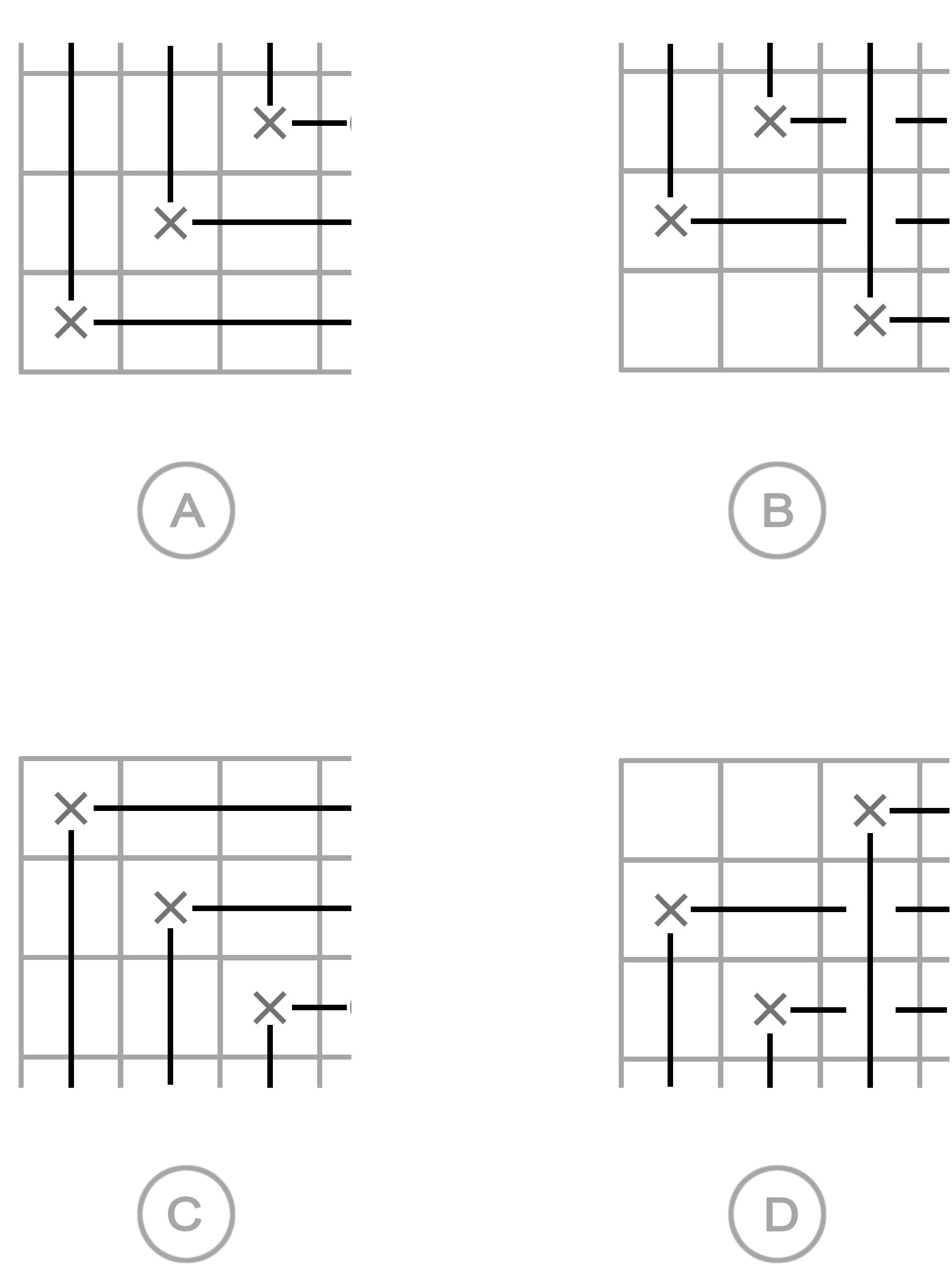}
\caption{Blocks A and C have markings on diagonal. B and D have $p-1$ markings on a diagonal and the last one in a corner}\label{four}
\end{center}

\end{figure}

Given a grid $D$ of a link $L$, we can construct grids of $p$-cables from the grid by transforming a single square to a $p \times p$ block, so that, empty squares are transformed to empty blocks and marked squares are transformed into certain types of (A, B, C and D  as shown in the Figure ~\ref{four})  of blocks. Now to ensure that we get a cable, we need to restrict allowed block types for certain corners. For top right and bottom left corner corners we use blocks A or B. For top left and bottom right corners we use block C or D. \textbf{ For proving our key theorem, all X marked squares, O marked South-West and O marked North-East corners, will be replaced by blocks of type A}. However, then we may not get a cable for $p>2$. We will rectify this situation in ~\ref{entwo}.   \\

Using the prescription, we can obtain the grids $D_p$ of the cable of topological link type ${L}_{p,q}$. The cabling parameter with respect to the blackboard framing of the constructed cables will be denoted by $q_B$. Here, $q$ and $q_B$ are related by $q=p \cdot wr(D) + q_B$. \\

\begin{figure}
\begin{center}

\includegraphics[scale=0.50]{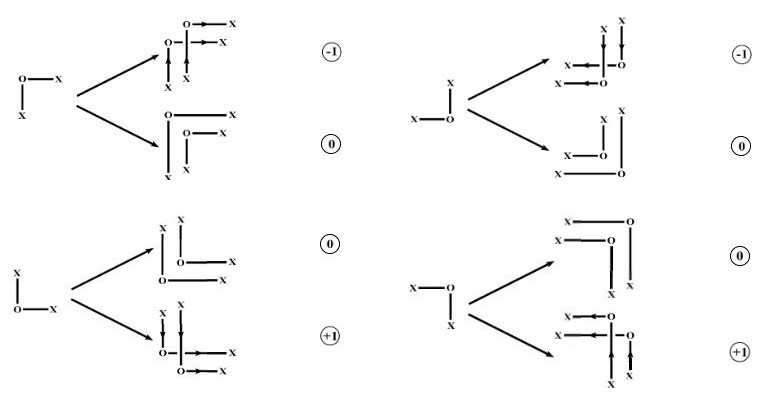}
\end{center}

\caption{Contributions to cabling coefficient for different block types and corners}\label{twists}

\end{figure}

In the grid diagram $q_B$ can be computed by summing up contribution from each block. In a single block the contribution to the cabling coefficient depends both on the corner type and block type. Figure  \ref{twists} computes all possible local contributions for $2$-cables.  \\

 The Legendrian and transverse link types represented by the grid may depend on the choice of the blocks. $\mathcal{L}_{p,q}$ ($\mathscr{L}_{p,q}$) will denote the Legendrian (transverse) link types resulting from the construction. We will show in Proposition \ref{classinvs} that classical invariants are independent of the choices of blocks and grid diagram of $L$.  Now we will determine for what values of cabling parameter $q$, we can obtain  Legendrian/transverse representatives of ${L}_{p,q}$. We first examine the case $p=2$ and then we will extend the results for $p>2$.  

\begin{remark}
It is possible to make arbitrary choices of blocks to get the grid of certain satellites. However, it is not so clear what kind of satellite we can obtain from this construction.
\end{remark}

\subsubsection{Generating 2-cables} 

For 2-cables, we will only have blocks of type A and C. We will try to use stabilizations to obtain a wide range of twisting coefficients. 

\begin{figure}[h!] 
\centering 
\includegraphics[width=0.35\textwidth]{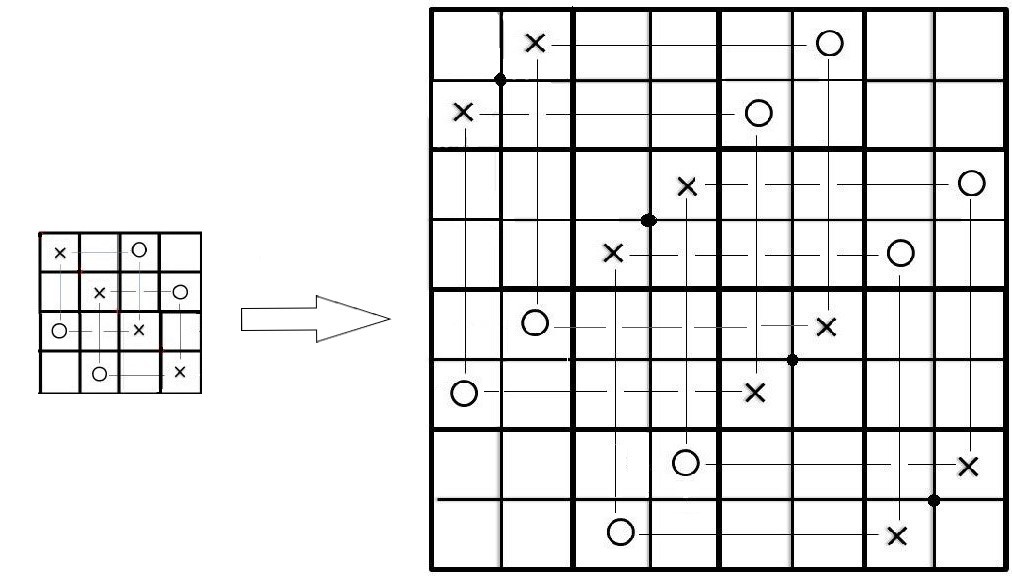}    
\caption{Generating grid of a 2 cable}
\end{figure}

\begin{prop}\label{toptwist}

The grid $D_{2}$ can be constructed from grid $D$, link $L$, by following the above procedure to represent a cable link  $L_{2,q}$  if $ 2 wr(D) - x_{SE} - x_{NW} \geq q \geq 2 wr(D) - o_{SE}- o_{NW} - x_{SE} - x_{NW}$. 

\end{prop}

\begin{proof}
To determine the cabling coefficient, we need to keep track of contribution for each type of corners and block type [See Figure~\ref{twists}] as well as signs of crossings. If we replace each markings by blocks of type $A$, we get $q_B= - o_{SE}- o_{NW} - x_{SE} - x_{NW}$. Therefore, $q= 2 wr(D)+q_B = 2 wr(D) - o_{SE}- o_{NW} - x_{SE} - x_{NW}$ which is the lower bound. Then, by replacing O-SE and O-NW corners by blocks of type C,  we get the upper bound.
\end{proof}

\begin{prop}\label{legcon2}
  If $D$ represents a Legendrian link $L$, then we can make certain stabilizations in $D$ followed by appropriate choices of blocks to construct a grid $D'_{2}$ representing a Legendrian link $\mathcal{L}_{2,-q}$ as long as $q \leq 2 wr(D) - x_{SE} - x_{NW}$.
\end{prop}

\begin{proof}

First, we realize that mirroring changes the sign of $q$ in the represented Legendrian link. 

Since $o_{SE}+ x_{SE} \geq 1$ for any grid and X:SE and O:SE stabilizations don't change the Legendrian link type; we can carry out the procedure of replacing a square by blocks after performing repeated stabilizations on those corners (X:SE stabilization on X:SE corners and O:SE stabilization on O:SE corners) to decrease $q$ by any arbitrary number. 

\end{proof}

\subsubsection{Generating $p$-cables for $p>2$}\label{entwo}
When $p>2$, using block A for NW and SE X corners induces a half full twist in the satellite. To get an integer value of the twisting parameter, we need to perform a stabilization on those X corners in the grid diagram $D$  before replacing the X marked squares by block A [See Figure~\ref{nCab}]. Let $D'$ denote the grid diagram obtained from $D$ using stabilizations. $D'$ represents the same Legendrian link type $L$ since these stabilizations don't change the underlying Legendrian link type. Again for O markings; we are allowed to use any blocks for these types of corners. There are obvious extensions of the results in the previous section. First, we state the extension of Proposition \ref{toptwist} -

\begin{figure}
\begin{center}

\includegraphics[scale=0.60]{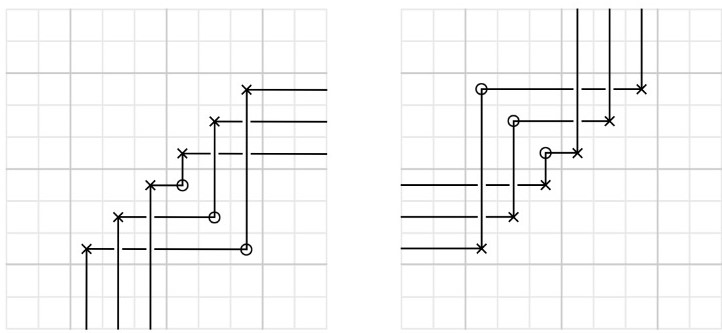}
\end{center}

\caption{Applying X:SE stabilization for  X:SE corners and X:NW stabilization for X:NW corners}\label{nCab}

\end{figure}

\begin{prop}\label{nleg}

The grid $D_{p}$ can be constructed from grid $D$, representing  link $L$, by following the above procedure to represent a cable link  ${L}_{p,q}$ if $ p(wr(D) - x_{SE} - x_{NW}) \geq q \geq p(wr(D) - x_{SE} - x_{NW}) - o_{SE}- o_{NW} $. 

\end{prop}
\begin{proof}
We first applying X:SE stabilization for  X:SE corners and X:NW stabilization for X:NW corners before block replacement. Each X: SE and X:NW corner contributes $-p$ to $q_B$ [Figure~\ref{nCab}]. Using type D blocks for O:SE and O:NW corners yields $-1$ contribution from each of them and gives us the lower bound. If we use blocks of type C instead for those corners, we get the upper bound.
\end{proof}

For Legendrian links, we are able to obtain the following class of cables.\\ 

\begin{prop}
  If $D$ represents a Legendrian link $L$, then we can construct a grid $D'_{p}$ representing a Legendrian link $\mathcal{L}_{p,q}$ as long as $q \geq p( 2n(L) + tb(L) )$ where $n(L)$ is the minimum grid number of a gid representing the Legendrian link $L$.
\end{prop} \label{construct}
\begin{proof}
We get the upper bound for $q$ from   Proposition \ref{nleg}. We can make a torus translation to make sure the grid has at least one SE or NW O corner. Then we can apply  O-SE stabilizations or O-NW stabilizations to decrease the value of $q$ while keeping the Legendrian link type intact. In the constructed grid $D'_{p}$ representing the Legendrian link $\mathcal{L}_{p,q}$, we need to switch the sign of $q$ as it is obtained by mirroring. This way, we can construct $\mathcal{L}_{p,q}$ for $q \geq p (  -wr(D) + x_{SE} + x_{NW})= p [  tb(L) + \frac{1}{2}(x_{NE}+o_{SW}+x_{SW}+o_{NE}) + x_{SE} + x_{NW} ] $. Now, assuming we started with a grid diagram with minimum grid number $n(L)$, we get an upper bound $ \frac{1}{2}(x_{NE}+o_{SW}+x_{SW}+o_{NE}) + x_{SE} + x_{NW} < 2n(L)$.
\end{proof}

As a corollary, we can also construct grid diagrams of transverse cables with the same bounds.
\begin{cor}\label{trans}
  If $D$ represents a transverse link $L$, then we can make certain stabilizations in $D$ followed by appropriate choices of blocks to construct a grid $D'_{P}$ representing a transverse link $\mathscr{L}_{p,q}$ as long as $q \geq p( 2n(L) + tb(L) )$.
\end{cor}

\begin{remark}
The lower bound that we obtained for $\frac{q}{p}$ is a very weak approximation but it will allow us to concretely list the cables that are transversely non-simple in Section 4.1. In practice, one may as well work with non minimal grid diagrams as in Figure \ref{untb}. This maybe particularly useful while dealing with braids.
\end{remark}
Now, we will determine the Thurston-Bennequin and rotation numbers of the constructed Legendrian cables. 

\begin{figure}
\begin{center}

\includegraphics[scale=0.12]{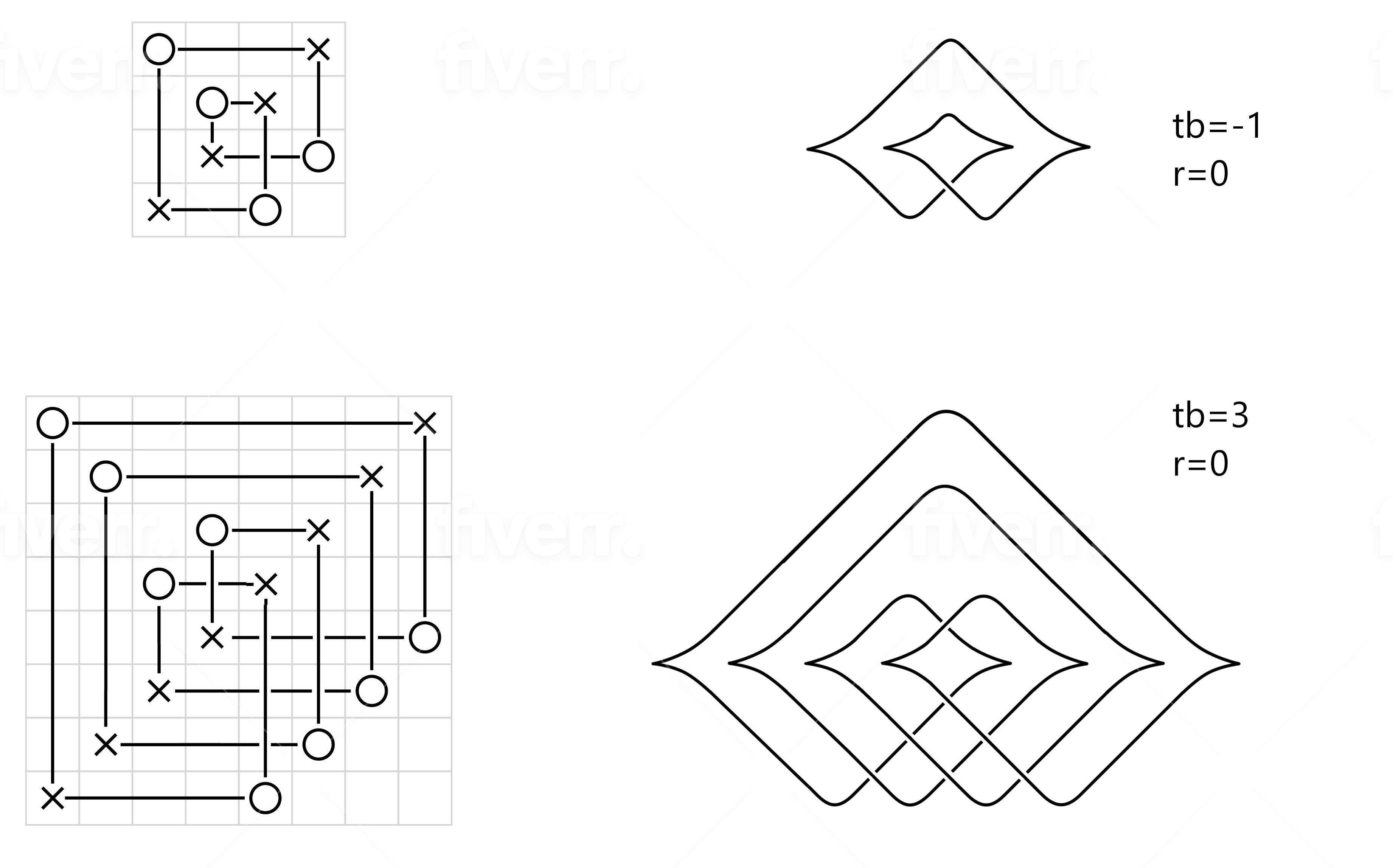} 
\caption{At the top we have a grid diagram of a Legendrian unknot with $tb=-1$ and $r=0$. At the bottom, the grid diagram (2,5) cable of Legendrian unknot having $tb=3$ and $r=0$ is obtained using our prescription}\label{untb}
\end{center}
\end{figure}

\begin{prop} \label{classinvs}

Let $\mathcal{L}_{p,q}$ be the Legendrian cable constructed using the prescription. Then,
$tb(\mathcal{L}_{p,q})= p \cdot tb(L) + (p-1)q $ and
$r(\mathcal{L}_{p,q})=p \cdot r(L)$.

\end{prop}

\begin{proof}
Let $x_{SE}$, $x_{SW}$, $x_{NE}$ and $x_{NW}$ denote the number of respective x corners in the grid $D_{0}$ obtained from $D$ right before block replacement (Similarly for $o_{SE}$, $o_{SW}$, $o_{NE}$ and $o_{NW}$). From the formula of $tb$ in terms of grid diagram, we have \[p \cdot tb(L)= - p \cdot wr(D_0)- \frac{1}{2}(p \cdot  x_{NE} + p \cdot o_{SW}+ p \cdot  x_{SW} + p \cdot  o_{NE}) .\]

Now using, $q_B$ (the cabling parameter of $D'_p$ with respect to the blackboard framing) we have $-q=p \cdot wr(D_0) + q_B$ (Since $D'_p$ represents $L_{p,-q}=m(L_{p,q})$). We make the following computation (See Figure \ref{untb}),
\begin{align*}
 tb(\mathcal{L}_{p,q})= - wr(D'_{p})- \frac{1}{2}(p \cdot  x_{NE} + p \cdot o_{SW}+ p \cdot  x_{SW} + p \cdot  o_{NE})\\
 = - wr(D'_{p}) + p \cdot wr(D_0) +  p \cdot tb(L) \\
 = - [ p^2 wr(D_0) + (p-1) q_B] + p \cdot wr(D_0) +  p \cdot tb(L)\\
 = - (p-1) [ p \cdot wr(D_0) + q_B] +  p \cdot tb(L)\\
 =   p \cdot  tb(L) + (p-1)q.
 \end{align*}
Similarly, it is straight forward to see that 
\[r(\mathcal{L}_{p,q})= \frac{1}{2}(p \cdot x_{NE}+ p \cdot o_{SW}- p \cdot x_{SW}- p \cdot o_{NE}) \\
=\frac{p}{2}(x_{NE}+o_{SW}-x_{SW}-o_{NE}) \\
=p \cdot r(L) .\]

\end{proof}

\vspace{10mm}
\pagebreak

\section{Grid Homology}

We will consider several chain complexes associated to grid diagrams. These grid complexes provide a combinatorial approach to link Floer complexes defined using holomorphic theory \cite{HolomorphicLinkInvariants}. The reader is referred to \cite{Grid Homology for Knots and Links} for   
a comprehensive exposition.\\

A grid state $x$ for a toroidal grid diagram $D$ with grid number $n$ consists of
$n$ points in the torus such that each horizontal and each vertical circle contains precisely one element of $x$. The set of grid states for $D$ is denoted by $S(D)$. Equivalently, we can regard the generators as $n$-tuples of intersection points between the horizontal and vertical circles, such that no intersection point appears on more than one horizontal or vertical circle.\\

Given $x,y \in S(D)$, let $Rect(x,y)$ denote the space of embedded rectangles with the following properties. $Rect(x,y)$ is empty unless $x,y$ coincide at exactly $n-2$ points. An
element $r$ of $Rect(x,y)$ is an embedded disk in $\mathbb{T}$, whose boundary consists of four
arcs, each contained in horizontal or vertical circles; under the orientation induced
on the boundary of $r$, the horizontal arcs are oriented from a point in $x$ to a point
in $y$. The set of empty rectangles $r \in Rect(x,y)$ with $x \cap Int(r) = \phi$ is denoted
by $Rect^{o}(x, y)$. More generally, a path from $x$ to $y$ is a $1$-cycle $\gamma$ on $\mathbb{T}$ contained in the union of horizontal and vertical circles such that the boundary of
the intersection of $\gamma$ with the union of the horizontal curves is $y-x$ , and a domain
$\Delta$ from $x$ to $y$ is a two-chain in $\mathbb{T}$ whose boundary $\partial \Delta$ is a path from x to y.\\

The \textbf{unblocked grid complex}, $(GC^{-}(D), \partial_{\mathbb{X}}^{-})$, is defined \cite{Grid Homology for Knots and Links}  in the following way- \[ {GC}^{-}(D)= \ Free \ \mathbb{F}_{2}[V_{1},V_{2},...,V_{n}] \ module \ over \ grid \ states \ S(D) \] 
\[ \partial_{\mathbb{X}}^{-}x := \sum\limits_{y\in S(D)} \sum\limits_{r \in Rect^{o}(x,y), r\cap \mathbb{X}=\phi} V_{1}^{O_{1}(r)}...V_{m}^{O_{m}(r)} y \ \  \forall x \in S(D) \] \\

For sets $P,Q$ of finitely many points define, $\mathcal{J}(P,Q)=\frac{\mathcal{I}(P,Q)+\mathcal{I}(Q,P)}{2} $ where $\mathcal{I}(P,Q)$ counts ordered pairs of points $(a,b)\in P \times Q$ such that $b$ has both coordinates greater than $a$. The homological grading of a generator $x \in S(D)$ in this complex is given by Maslov grading which is defined as a function  $M: S(D) \rightarrow \mathbb{Z}$, \[ M(x) = \mathcal{J} (x - \mathbb{O},x - \mathbb{O})+1 \]  Multiplication by $V_{i}$ lowers the Maslov grading by $2$. $\partial_{\mathbb{X}}^{-}$ lowers Maslov grading by 1. Additionally, the complex comes with an Alexander grading  $A$. For a generator $x \in S(D)$, it is defined by the formula,
 \[ A(x) = \frac{\mathcal{J} (x - \mathbb{O},x - \mathbb{O})- \mathcal{J} (x - \mathbb{O},x - \mathbb{O})}{2} -\frac{n-1}{2} \]
 
 Multiplication by $V_{i}$ lowers the Alexander grading by $1$. $\partial_{\mathbb{X}}^{-}$ drops the Maslov grading by $1$ and preserves the Alexander grading.\\ 

The maps given by multiplication by $V_{i}$ and $V_{j}$ are chain homotopic in  $GC^{-}(D)$ if $O_{i}$ and $O_{j}$ are in the same link component. Therefore, ${GH}^{-}(L)$ can be thought of as a bi-graded $\mathbb{F}_{2}[V_{1},V_{2},...,V_{l}]$ module if the link $L$ represented by $D$ has $l$ components. ${GH}^{-}(L)$ is a link invariant (\cite{MOS}) and is isomorphic to link Floer homology (${HFL}^{-}$). \\

There are several collapsed complexes that we can construct from the $\mathbb{F}_{2}[V_{1},V_{2},...,V_{n}]$  module ${GC}^{-}(D)$ by setting some of the $V_{i}$'s equal to each other. The \textbf{collapsed link grid complex} is defined as $c{GC}^{-}(D): = \frac{{GC}^{-}}{V_{i_{1}}=V_{i_{2}}=...=V_{i_{l}}}$ , where $O_{i_{k}}$ is a $O$ marking belonging in the k'th link component. Its homology, $cGH^{-}(D)$ can be thought of as a $\mathbb{F}_{2}[U]$ module. In fact, it can be shown that $cGH^{-}(D) \cong (\mathbb{F}_{2}[U])^{2^{l-1}} \oplus Tor $ (Here $Tor$ is the torsion part).\\ 

The \textbf{simply blocked grid complex}, $\widehat{GC}(D)$ is defined as $\frac{c{GC}^{-}(D)}{U_{1}=0}$ and it has homology $\widehat{GH}(L)$ which is a link invariant and can be thought as a $\mathbb{F}_{2}$ module. There is also a \textbf{fully blocked grid complex} $\widetilde{GC}(D)$, defined as $\frac{GC^{-}}{U_{1}=U_{2}=..=U_{n}=0}$. Since the complex itself is a $\mathbb{F}_{2}$  module, its simplest to compute. Its homology $\widetilde{GH}(D) \cong \widehat{GH}(L) \otimes W^{n-l}$, where $W$ is a $2$ dimensional graded vector space with generators in bi-gradings $[0,0]$ and $[-1.-1]$.\\

The element $x^{+} \in S(D)$, which consists of the intersection points at the upper right corners of the squares containing the markings X in $D$, is a cycle in $({GC}^{-}(D), {\partial}_{\mathbb{X}}^{-})$. The element $x^{-} \in S(D)$ consisting of the intersection points at the south west corners of X markings is also a cycle. If $L$ is the Legendrian link corresponding to the grid diagram $ D$, then we know that $D$ represents the topological link type of $m(L)$. The homology classes $[x^{+}], [x^{-}]  \in GH^{-}(m(L))$, denoted by $\lambda^{+}(D)$ and $\lambda^{-}(D)$ respectively, are called the Legendrian grid invariant of $D$. For the transverse push-off $\mathcal{T}$ of an oriented Legendrian link $L$, the transverse grid invariant $\theta^{-}(D)$ is defined to be $\lambda^{+}(L) \in GH^{-}(m(L)) $. The following proposition states that the homological class is a well-defined invariant of Legendrian and transverse link types. 

\begin{prop}\cite{theta} Let $D$ and $D'$ be two grid diagrams corresponding to Legendrian link $L$ (similarly transverse link $\mathcal{T}$), then there is an isomorphism \\

$\phi : GH^{-}(D) \rightarrow GH^{-}(D')$\\

such that $\phi(\lambda^{+}(D))=\lambda^{+}(D')$ and $\phi(\lambda^{-}(D))=\lambda^{-}(D')$  (similarly $\phi(\theta^{-}(D))=\theta^{-}(D')$). 
\end{prop}

Therefore, we choose to write $\lambda^{+}(D)$ as $\lambda^{+}(L)$ and $\lambda^{-}(D)$ as $\lambda^{-}(L)$  when $D$ corresponds to Legendrian link type $L$. It should be noted that $\lambda^{+}(L)$ and $\lambda^{-}(L)$ are well defined only upto isomorphism in grid homology. Similarly, we will write $\theta^{-}(D)$ as $\theta^{-}(\mathcal{T})$ when $D$ corresponds to transverse link type $\mathcal{T}$. It is often more useful to consider the projection of $\theta(\mathcal{T})$ into $\widehat{GH}$, which we will call $\hat{\theta}(\mathcal{T})$. Projection of $\hat{\theta}(\mathcal{T})$. into $\widetilde{GH}(D)$ will be denoted as $\tilde{\theta}(\mathcal{T})$. It can be showed that $\hat{\theta}(\mathcal{T})=0$ if and only if  $\tilde{\theta}(\mathcal{T})=0$.\\

In our discussion, we will consider a different collapsed grid complex that we will call \textbf{fully collapsed grid complex}. It will be denoted by $\mathscr{C}(D)$ for a grid diagram $D$. 

\begin{definition}\label{definec}
 Let $D$ be a grid diagram of $m(L)$ for some link $L$. Define $\mathscr{C}(D)$ as $\mathbb{F}_{2}[U]$ module over grid states $S(D)$ and, 
  \[ {\partial}x := \mathlarger{\sum\limits_{y\in S(D)} \sum\limits_{r \in Rect^{o}(x,y), r \cap \mathbb{X} = \phi}} U^{O(r)} y \ \  \forall x \in S(D) \]
\end{definition}

The homology of this complex is not a link invariant, but the following proposition (Similar to Lemma 14.1.11 in \cite{Grid Homology for Knots and Links}) gives its relation with collapsed grid link complex.
\begin{prop}\label{quasi}

There is a quasi-isomorphism 
$\mathscr{C}(D) \cong cGC^{-}(D) \otimes W^{n-l}$.
\end{prop}

\begin{proof}
Assume that $O_{i}$ and $O_{j}$ belong in the same link component. Let us consider the short exact sequence\\

\begin{tikzcd}
 0 \arrow{r} & c{GC}^{-}(D) \arrow{r}{V_{i}- V_{j}}  & c{GC}^{-}(D)  \arrow{r}  & \frac{c{GC}^{-}(D)}{V_{i}- V_{j}} \arrow{r} & 0
\end{tikzcd}\\

We know that the map given by multiplication by $V_{i}- V_{j}$ is chain homotopic to $0$. Also , multiplication by  $V_{i}- V_{j}$ lowers  Maslov grading by $2$ and Alexander grading by $1$. Therefore, mapping cone is quasi-isomorphic to $c{GC}^{-}(D) \oplus c{GC}^{-}(D)[-1,-1]$. So we have a quasi-isomorphism $\frac{c{GC}^{-}(D)}{V_{i}- V_{j}} \rightarrow c{GC}^{-}(D) \otimes W $ , and the conclusion follows by iteration.
 
\end{proof}

We will need the following property of the distinguished cycle later in the next section.
\begin{prop}
The class $[x^+] \in H_{*}( \mathscr{C}(D))$ is in the $U$-image if and only if $\hat{\theta}(L)=0$.
\end{prop}
\begin{proof}
Consider the short exact sequence\\

\begin{tikzcd}
 0 \arrow{r} & \mathscr{C}(D) \arrow{r}{U} & \mathscr{C}(D) \arrow{r} & \frac{\mathscr{C}(D)}{U} \arrow{r} & 0 
\end{tikzcd}\\

Therefore, from the induced long exact sequence, we can infer that if the class $[x^+]$ is in $U$-image then the projection of $[x^+]$ in  $\frac{\mathscr{C}(D)}{U}$ is 0. Now notice that $\frac{\mathscr{C}(D)}{U} \cong \widetilde{GC}(D)$ and, the projection of $[x^+]$ there is $\tilde{\theta}(L)$. So we get  $\tilde{\theta}(L)=0$. This implies $\hat{\theta}(L)=0$. Conversely, $\hat{\theta}(L)=0$ implies $\tilde{\theta}(L)=0$. Hence, the short exact sequence implies that $[x^+] \in H_{*}( \mathscr{C}(D))$ is in the $U$-image.   

\end{proof}

\pagebreak

\section{Cables and transverse invariant}

Now we are ready to use tools of grid homology to study the constructed cables in Section \ref{constructed_grids}. We first define a change of variable in the fully collapsed complex that will be useful for relating the link complex with its cable complex.
\begin{definition}
 Let $D$ be a grid diagram of $m(L)$ for some link $L$ and $p \in \mathbb{N}$.  Define $p\mathscr{C}$ as $\mathbb{F}_{2}[U]$ module over grid states $ S(D)$ and \vspace{5mm}\\
 ${\partial}_{p\mathscr{C}}x := \mathlarger{\sum\limits_{y\in S(D)} \sum\limits_{r \in Rect^{o}(x,y), r \cap \mathbb{X} = \phi}} U^{p O(r)} y $ for a $x \in S(D)$\\ 
\end{definition}

Algebraically  $p\mathscr{C}$ is obtained from $\mathscr{C}$ by a change of variable. Therefore it inherits a new Alexander grading $\mathsf{A}$ satisfying, $\mathsf{A}(x)=pA(x)$ for $x \in S(D)$ and $\mathsf{A}(U)=-1$. The Maslov grading $M$ from $\mathscr{C}$, can be adapted as $\mathsf{M}$ in $p\mathscr{C}$ so that $\mathsf{M}(x)=pM(x)$ for $x \in S(D)$ and $\mathsf{M}(U)=-2$.  ${\partial}_{p\mathscr{C}}$ preserves the $\mathsf{A}$ and decreases $\mathsf{M}$ by $p$.

Now let us consider the grid $D_{p}$ of the $p$-cable constructed from $D$ using  the prescription given in \ref{constructed_grids}. $\mathscr{L}_{p,q}$ ($\mathcal{L}_{p,q}$)  will refer to transverse (Legendrian) link type corresponding to that grid. $L_{p,q}$ will denote the underlying link type.
\begin{figure}

\begin{center}
\includegraphics[scale=0.1]{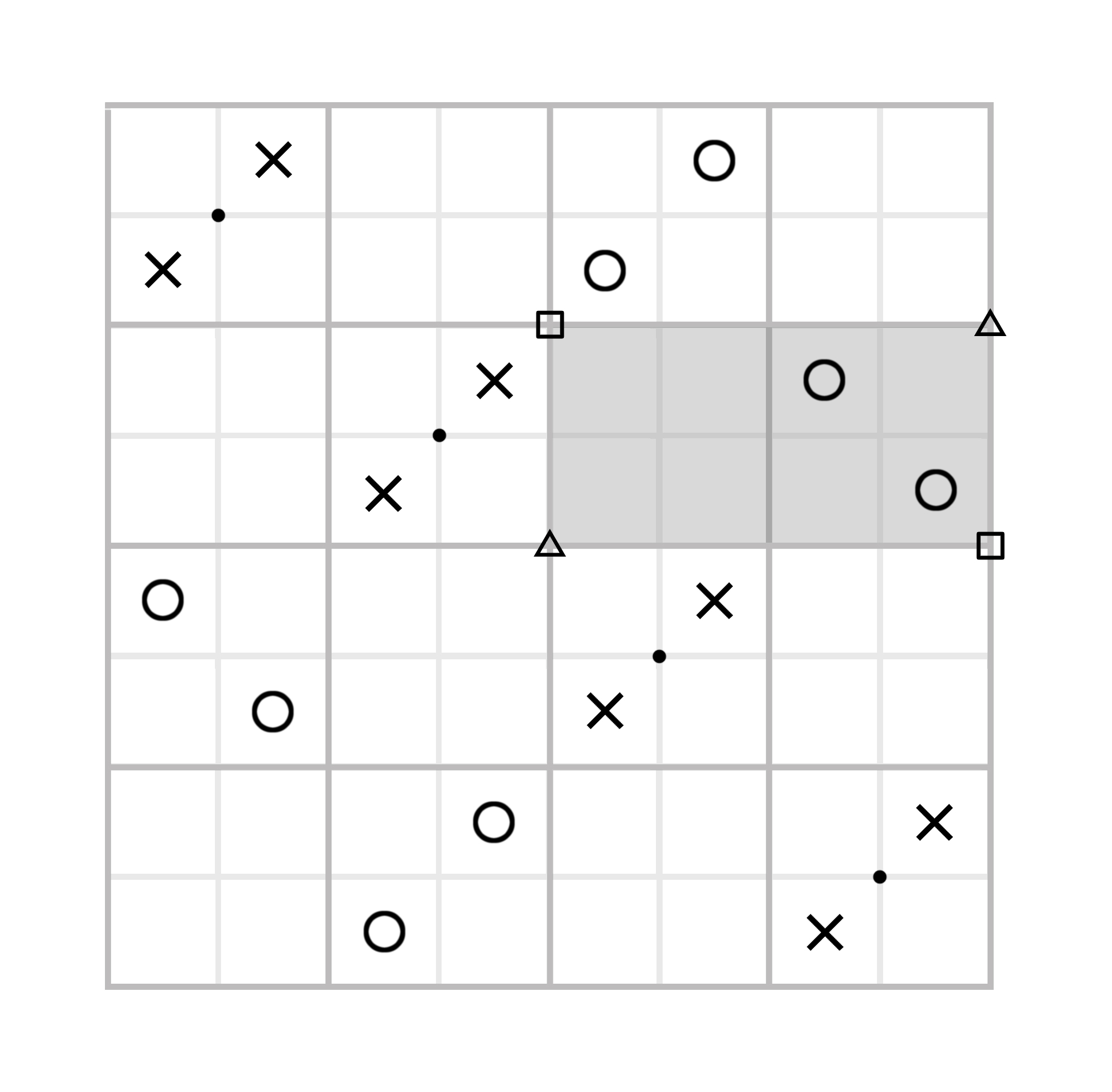} 
\end{center}

\caption{States of subcomplex $\mathcal{K}$} \label{cblock}
\end{figure}
Define $i : p\mathscr{C} \rightarrow \mathscr{C}(D_{p})$ to be the $\mathbb{F}_{2}[U]$ module map that takes a generator state $x$ in parent grid $D$ to a state in the cable grid $D_{p}$ obtained by taking union of  North-East corners of X in the middle of each block and $x$. Let $\mathcal{K}$ be the sub-module of $\mathscr{C}(D_{p})$ generated by all the states that contain North-East corner of X in the middle of each block [See Figure~\ref{cblock}].

\begin{prop}
The map $i$ is a injective chain map and $\mathcal{K}$  is a  subcomplex of $\mathscr{C}(D_{p})$ isomorphic to $p\mathscr{C}$ .
\end{prop}
\begin{proof}
It follows from the definition that $i (p\mathscr{C}) = \mathcal{K}$. Now to prove that $i$ is a chain map, we need to verify that $i {\partial}_{p\mathscr{C}}= {\partial}_{\mathscr{C}(D_{p})} i$. Lets take  states $\mathsmaller{\square},\vartriangle \in p\mathscr{C}$ such that $ {\partial}_{p\mathscr{C}}(\vartriangle)= U^{pk} \ \mathsmaller{\square}+..$ [as depicted in Fig \ref{cblock}]. This implies  $ i({\partial}_{p\mathscr{C}}(\vartriangle))= U^{pk} \ i(\mathsmaller{\square}) + ... \ $. Also, we have ${\partial}_{\mathscr{C}(D_{p})} ( i( \vartriangle))  = U^{pk} \ i(\mathsmaller{\square})+ ... \ $ because the shaded rectangle contains $p$ times manys Os in the cable grid. Since there is no rectangle coming out the special points of $\mathcal{K}$, any rectangle coming out of $i(\vartriangle)$ must join it with another state of the form $i(\mathsmaller{\square})$ for some $\mathsmaller{\square}$. Hence, the map $i$ satisfies $i {\partial}_{p\mathscr{C}}= {\partial}_{\mathscr{C}(D_{p})} i$. Also since $i$ is an injective chain map, it follows that $\mathcal{K}$ of is a  subcomplex of $\mathscr{C}(D_{p})$ isomorphic to $p\mathscr{C}$.   

\end{proof}
 
\begin{prop}\label{graddeg}

The map $i$ sends the distinguished cycles $[x^+]$ and $[x^{-}]$ in $p\mathscr{C}$ to the distinguished cycles $[x^+]$ and $[x^{-}]$ respectively in $\mathscr{C}(D_{p})$. It shifts the Alexander grading by $\frac{(p-1)(q-1)}{2}$ and Maslov grading by $(p-1)(q-1)$.
\end{prop}

\begin{proof}

It is obvious from the construction that $i$ sends the distinguished states $x^+$ and $x^{-}$ in $p\mathscr{C}$ to the  distinguished states $x^+$ and $x^{-}$ respectively in $ \mathscr{C}(D_{p})$. Also it is easy to see that $i$ respects relative Alexander and Maslov grading. Hence, we just need to compute the Alexander and Maslov grading difference of the distinguished state in the respective complexes. Using 12.7.5 in \cite{Grid Homology for Knots and Links}, it is equal to   \[ \mathsmaller{ A(i(x^+))- \mathsf{A}(x^+)= \frac{sl(\mathscr{L}_{p,q})+1}{2}-\frac{p(sl(L)+1)}{2}=\frac{sl(\mathscr{L}_{p,q})- p \cdot sl(L)- (p-1) }{2} . }\] To compute this quantity lets assume $L$ has braid representative $\beta_{L}$ with index $N$ and that $\mathscr{L}_{p,q}$ has $r$ twists with respect to blackboard framing. Then, $q=p\cdot wr(\beta_L)+r$ and $\mathscr{L}_{p,q}$ has a braid representative $\beta_{\mathscr{L}_{p,q}}$ with index $Np$ and  $wr(\beta_{\mathscr{L}_{p,q}})= p^{2}\cdot wr(\beta_L) + r(p-1) $. We also know that for braid $\beta$ of index $n$, $sl(\beta)=wr(\beta)-n$. Hence, it is equal to - \[ \mathsmaller{ \frac{(wr(\beta_{\mathscr{L}_{p,q}})-Np)-p(wr(\beta_L)-N)- (p-1)}{2} = \frac{(p-1)(p \cdot wr(\beta_L)+r-1)}{2} = \frac{(p-1)(q-1)}{2}. } \]   
 
 Similarly, $ M(i(x^+))- \mathsf{M}(x^+)= (sl(\mathscr{L}_{p,q})+1)-p(sl(L)+1)= (p-1)(q-1)$.
\end{proof}

\begin{prop}\label{onedir}
 If $\hat{\theta}(L)=0$ then $\hat{\theta}(\mathscr{L}_{p,q})=0$
\end{prop}

\begin{proof}

 $\hat{\theta}(L)=0$ implies that $[x^+]$ is in the $U$-image in the homology of complex $p\mathscr{C}$ i.e., $[x^+]=Uy$ for some $y$. Then, $i_{*}([x^+])=i_{*}(Uy)=Ui_{*}(y)$ is also in $U$-image in the homology of complex $\mathscr{C}(D_{p})$. So, it follows that $\hat{\theta}(\mathscr{L}_{p,q})=0$.      

\end{proof}

\begin{figure}[!tbph]
\begin{center}

\includegraphics[scale=.06]{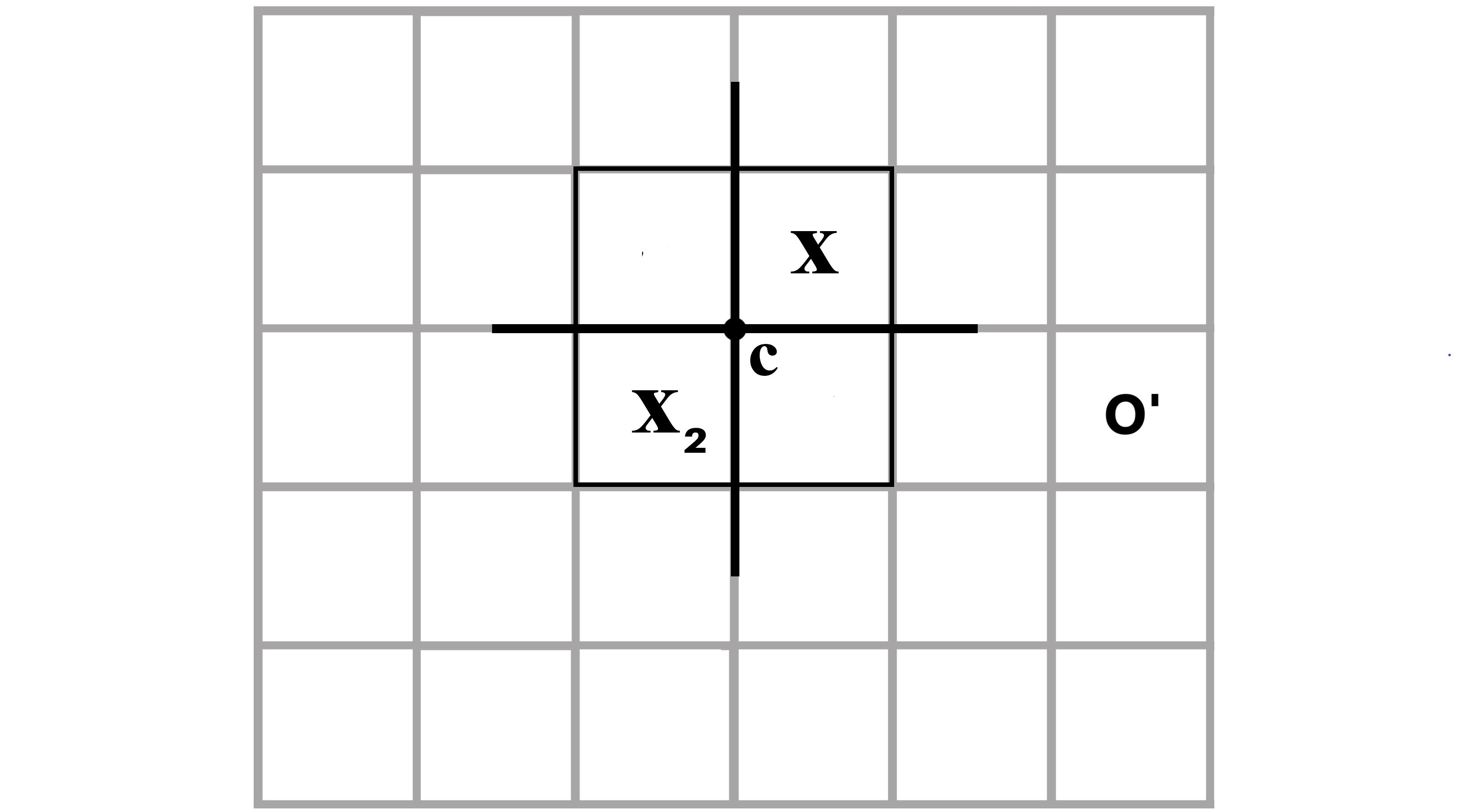}
\end{center}
\caption{Special point $c$ and markings around it}\label{special}
\end{figure}

\begin{theorem}\label{inclusionhom1}

Let ${L}_{p,q}$ be the cable constructed using the prescription of section 2 for $\frac{q}{p}$ sufficiently large and $i: p\mathscr{C} \rightarrow \mathscr{C}_{{L}_{p,q}} $ be the constructed inclusion map. $i([x^{\pm}]_{p\mathscr{C}})$ is in the $U$ image if and only if  $[x^{\pm}]_{\mathscr{C}_{{{L}_{p,q}}}}$ is in the $U$-image.
\end{theorem}

\begin{proof}

Let $ \tilde{i}: \frac{p\mathscr{C}}{U} \rightarrow \frac{\mathscr{C}_{L_{p,q}}}{U}$ be the induced map on quotients. Then it is enough to show that  $\tilde{i}([x^{\pm}]) =0$ iff $[x^ {\pm}]=0$ in $\frac{\mathscr{C}_{L_{p,q}}}{U}$. We first show it for the distinguished cycle $[x^+]$. The same argument also proves the assertion for  $[x^-]$.\\

Lets consider one of the $n \times n$ block in the grid [See Figure~\ref{special} ]. There are two X markings inside the block around the special point $c$. The north-east square is marked with X; the south-west square is marked with $X_{2}$, and they intersect at $c$. Let $O'$ be the O marking in the row containing $X_{2} $.    We will write, $ \frac{\mathscr{C}(D_{p})}{U} = \mathcal{S}\oplus \mathcal{N}$, where $\mathcal{S}$ is a sub-module generated by all states that contain a special point $c$ and $\mathcal{N}$ is a sub-module generated by all states that don't contain $c$. Since there are no rectangles coming out of the special point $c$, $\mathcal{S}$ is a subcomplex as before. Therefore, the differential of the complex can be written as, $\partial = \begin{bmatrix}
    {{\partial}_{S}}^{S}      & {{\partial}_{N}}^{S} \\
     0      & {{\partial}_{N}}^{N} \\
   
\end{bmatrix}
$. 

\begin{figure}
\begin{center}

\includegraphics[scale=.26]{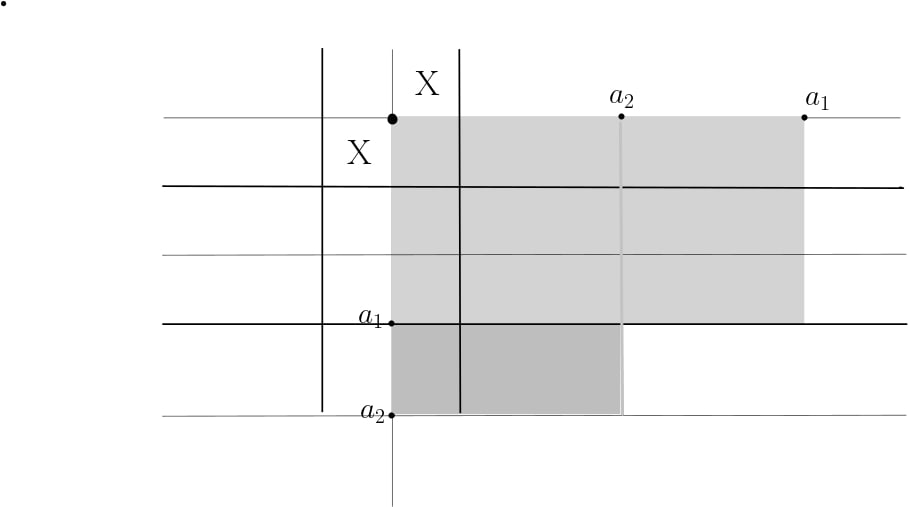}
\end{center}
\caption{ $a_i$'s are states in $\mathcal{N}$ that map to $x^+$ under ${{\partial}_{N}}^{S}$ }\label{cancellation}
\end{figure}

 So  the chain complex $(\frac{ \mathscr{C}(D_{p})}{U}, \partial)$ can be seen as $Cone({{\partial}_{N}}^{S})$. By repeating this construction for other special intersection points (in lieu of $c$), we observe that  $\frac{\mathscr{C}(D_{p})}{U}$ is obtained using the iterated mapping cone construction starting from $\frac{\mathcal{K}}{U}$. Then, at each level we have a short exact sequence \begin{equation*} 
\begin{tikzcd}
 0 \arrow{r} & \mathcal{S} \arrow{r}{ \tilde{i}} & Cone({{\partial}_{N}}^{S}) \arrow{r}{\pi} &  \mathcal{N} \arrow{r} & 0 
\end{tikzcd}
\end{equation*}
 
It induces the long exact sequence,

\begin{equation*} 
\begin{tikzcd}
 \cdots \arrow{r}{{{\partial}_{N}}^{S}} & H_{*}(\mathcal{S}) \arrow{r}{ \tilde{i}_{*}} & H_{*}(Cone({{\partial}_{N}}^{S})) \arrow{r}{\pi_{*}} &  H_{*}(\mathcal{N}) \arrow{r} & \cdots 
\end{tikzcd}
\end{equation*}
 So, the inclusion of $[x^+]$ at some level may become $0$ if and only if $[x^+]$ is in the image of $H({{\partial}_{N}}^{S}): H_{*}(\mathcal{N}) \rightarrow H_{*}(\mathcal{S})$ map. We claim that it is not possible if the blocks for X and O markings have been chosen using our prescription unless $[x^+]=0$ in the previous level also.\\

\begin{figure}
\begin{center}
\includegraphics[scale=.5]{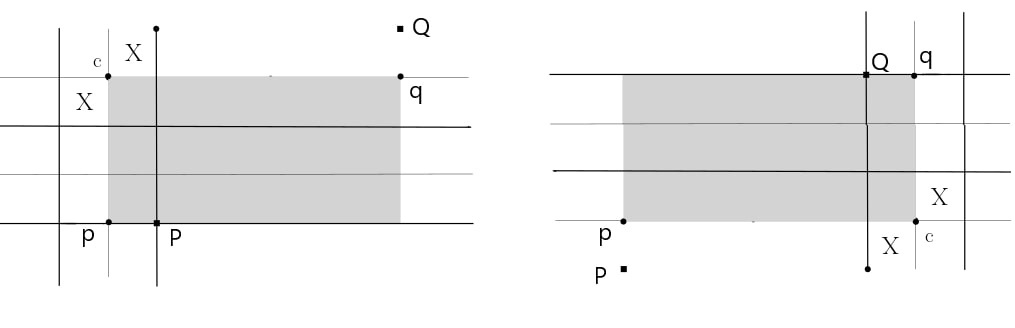}
\end{center}
\caption{ When c is the north east corner of the empty rectangle points p and q of the state $a_i$ are lifted to north east corners P and Q of the corresponding blocks in $\bar{a_i}$. When $c$ is the south east corner of the empty rectangle we lift p and q to south west  corners of the corresponding blocks P and Q. }\label{lifting}
\end{figure} 

Let $[x^+]= {\partial_{N}}^{S}( [a_1] + \cdots + [a_k])$. Where $a_i$ are states in $\mathcal{N}$ such that ${{\partial}_{N}}^{N}( \sum_{i} a_i)=0$. It follows that all but two points of $a_i$s are NE corners of X-markings. We call them $p$ and $q$. We can lift the states $a_i \in \mathcal{N}$  to $ \bar{a_i}  \in \mathcal{S}$, so that $\bar{a_{i}}$ differs from $a_{i}$ in only three points [See Figure \ref{lifting} and \ref{rightb}]. If $c$ is the north west corner of the empty rectangle then we lift p and q to north east corners of the corresponding blocks P and Q. When $c$ is the south east corner of the empty rectangle we lift p and q to south west  corners of the corresponding blocks P and Q. $ \bar{a_i} $ contains the special point c, it also contains P and Q (north east corners of the blocks containing the two points that are not north east corners of a X marking). Now, we have ${\partial_{S}}^{S}([\bar{a_i}])={\partial_{N}}^{S}([a_i])$. \\

Now, we make the key observation that there is no O or X marking in the rectangle bounded by $x^+$ and $ \bar{a_i}$ because otherwise that O-block is not a type A block as type $A$ blocks can only have O markings in the diagonal squares and the available (possibly non-empty) squares in the rectangle are non-diagonal (In Figures \ref{lifting2} and \ref{rightb} light yellow squares are non-diagonal)  $p \times p$ grid. But observe that if any O marking exist in the light yellow region, it will correspond to a north-east or south-west corner in the original grid $D$. So we can conclude that no O markings exist in the rectangle as north-east or south-west O corners were replaced by blocks of type A in our construction (See Figure \ref{rightb}).

\begin{figure}
    \centering
    \subfloat{{\includegraphics[scale=0.4]{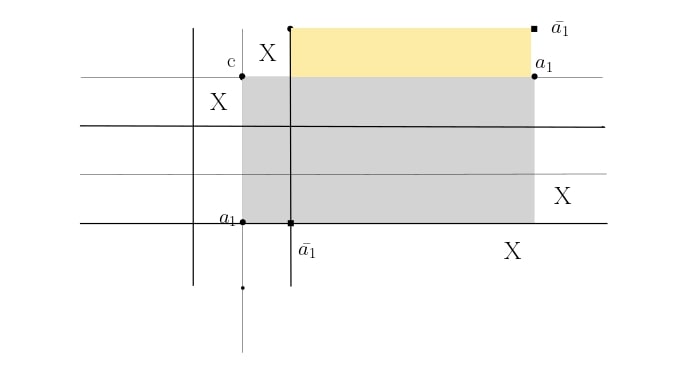} }}
    \qquad
    \subfloat{{\includegraphics[scale=0.4]{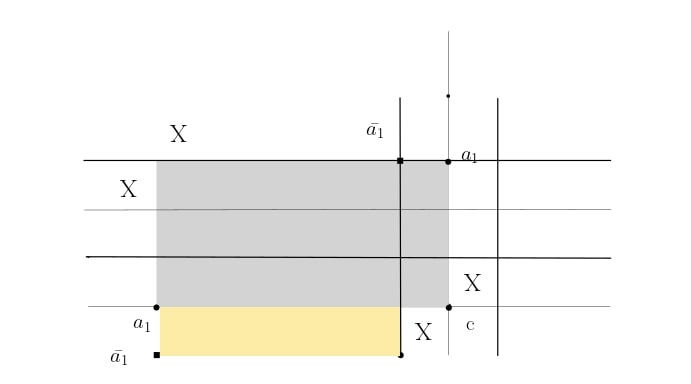} }}
     \caption{Light yellow regions contain no markings }\label{lifting2}
\end{figure} 

So we get  $[x^+]= {\partial_{S}}^{S}([\bar{a_1}]+\cdots+[\bar{a_k}])$.
By iterating this procedure for other special points, we reach the conclusion of our theorem.
\end{proof}

\begin{figure}
    \centering
   	\includegraphics[scale=0.1]{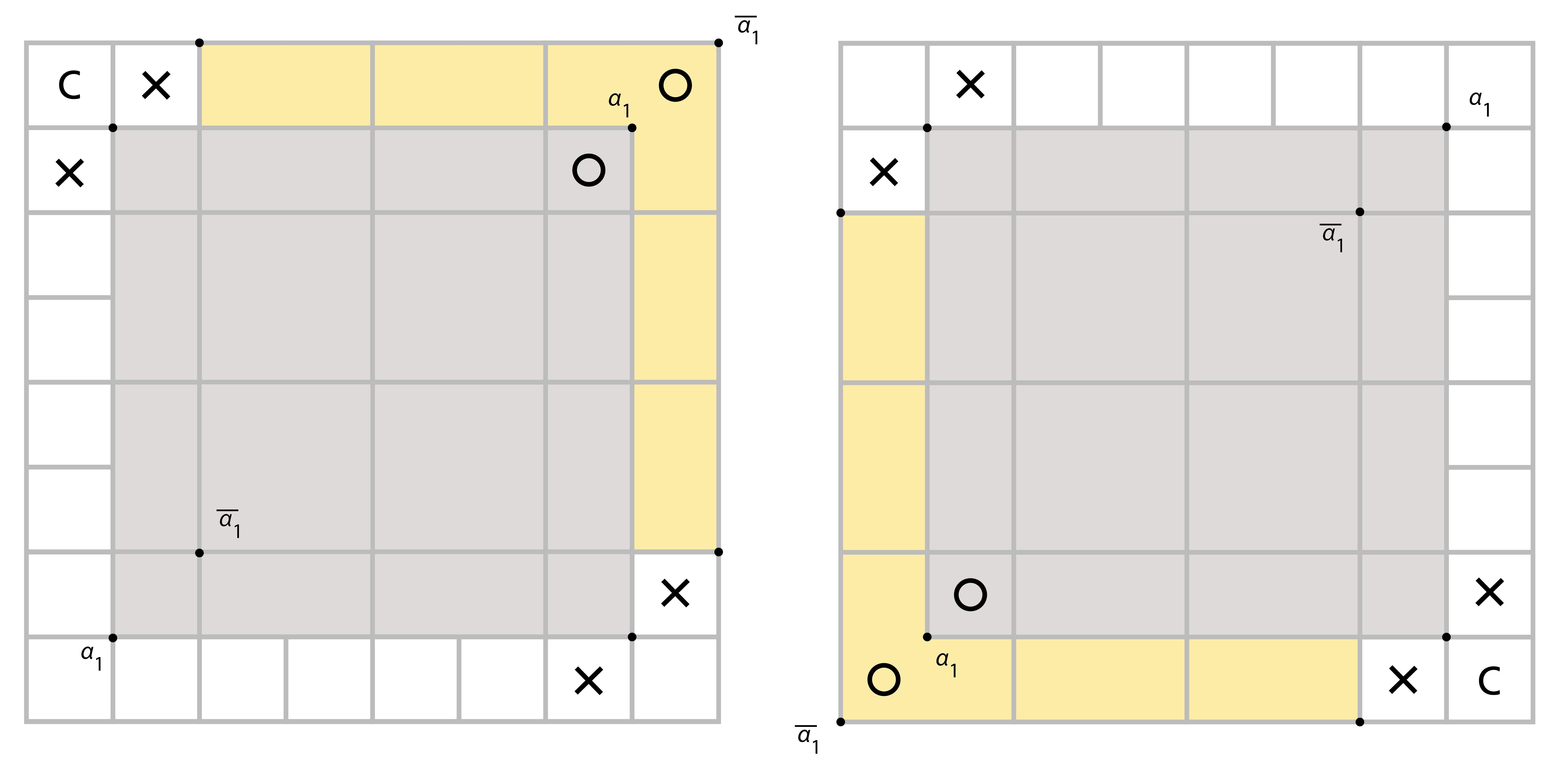}
     \caption{If Blocks of Type A is chosen for NE and SW O corners light yellow region containing O marking would imply existence of an O marking in the grey region}\label{rightb}
\end{figure} 

As a corollary, we obtain out key theorem -

\begin{proof}[Proof of Theorem \ref{mostimp}]
We already know one direction from Proposition \ref{onedir}. Now by Theorem \ref{inclusionhom1}, we know that $i$ induces inclusion (as $\mathbb{F}_{2}[U]$ module) on homology. So $\hat{\theta}(\mathscr{L}_{p,q})=0$ implies  $[x^+]$ is in $U$-image in $H_{*}(\mathscr{C}(D_{p}))$ and hence in $H_{*}(p\mathscr{C})$. It follows that $\hat{\theta}(L)=0$.  
\end{proof}

\vspace{5mm}
Now, for a Legendrian link $L$, instead of looking at $\lambda^{+}(L)$ or $\lambda^{-}(L)$ individually, it is more useful to consider the sum  $\lambda^{+}(L)+\lambda^{-}(L)$ that will be denoted by $\eta(L)$.  Its projection in the hat complex will be denoted by $\hat{\eta}(L)$.

\begin{proof}[Proof of Theorem \ref{chekanov}]

We want to show that $\hat{\eta}(L)=0$ if and only if $\hat{\eta}(\mathcal{L}_{p,q})=0$ for the mentioned values of $p$ and $q$.
First, we need to show that $\hat{\eta}(L)=0$ if and only if its projection $\hat{\eta}'(L)$ to the fully collapsed complex is $0$. Actually, this is true for any homology class. Suppose $D$ is a grid diagram for $L$. Let us consider the short exact sequence in \ref{quasi} again.

\vspace{5mm}
\begin{tikzcd}
 0 \arrow{r} & {GC}^{-}(D) \arrow{r}{V_{i}- V_{j}}  & {GC}^{-}(D)  \arrow{r}  & \frac{{GC}^{-}(D)}{V_{i}- V_{j}} \arrow{r} & 0
\end{tikzcd}\\

From the induced long exact sequence, we can infer that projection of any homology class $\alpha=[\xi]$ is $0$ then $\xi$ is in the image of $V_{i}- V_{j}$ which implies $\alpha=[\xi]=0$ since $V_{i}- V_{j}$ is null-homotopic. Conversely if $\alpha=0$ then obviously its projection is $0$. Iteration of this argument proves our claim.\\  

 Now we observe that, we could have used the argument of Theorem \ref{inclusionhom1} for the class $[x^+]+[x^-]$ to prove that  $i([x^+]+[x^-])$ is in the $U$-image if and only if $[x^+]+[x^-]$ is in the $U$-image. Since $\eta(L)= [x^+]+[x^-]$, the conclusion follows.
\end{proof}

\subsection{Examples of Legendrian and transversely non-simple links}

Now, let $K$ and $K'$ be two transverse links with same topological type and self-linking number such that, $\hat{\theta}(K)=0$ and $\hat{\theta}(K') \neq 0$. By our construction, $K_{p,q}$ and ${K'}_{p,q}$ also represent transverse links with same topological type and self-linking number but, they are not isotopic as $\hat{\theta}$ vanishes for only one of them. So we can combine our result with the already known examples to generate various infinite families of transversely non-simple link type. The following proposition gives an example -

\begin{prop}

The topological link type ${m(10_{132})}_{p,q}$ is transversely non-simple for $\frac{q}{p}>17$.
\end{prop}
\begin{proof}

 $m(10_{132})$ has transverse representatives $\mathcal{T}_{1}$ and $\mathcal{T}_{2}$ with same self-linking number $-$ such that $\hat{\theta}(\mathcal{T}_{1})=0$ and $\hat{\theta}(\mathcal{T}_{2})\neq 0$. Further, they are represented by grid number $9$ (\cite{Grid Homology for Knots and Links} Proposition 12.6.3). So we can construct cables for $\frac{q}{p}>17$ using Corollary \ref{trans}.  The conclusion follows.
\end{proof}

In the same vein, further examples can be obtained for cables of $m(10_{161})$,$m(7_{2})$ etc.\\
  
We know that vanishing of $\hat{\eta}$ distinguishes Chekanov pair in knot type $m(5_{2})$. The following proposition shows that some of its cables are also Legendrian non-simple. 

\begin{prop}

The topological link type ${m(5_{2})}_{p,q}$ is Legendrian non-simple for $\frac{q}{p} > 15$.
\end{prop}

\begin{proof}
There are Legendrian representatives $\mathcal{K}$ and $\mathcal{K}'$ of $m(5_{2})$ with $tb=1$ and $r=0$ such that $\hat{\eta}(\mathcal{K})=0$ and $\hat{\eta}(\mathcal{K}') \neq 0$ (\cite{Grid Homology for Knots and Links}, Proposition 12.4.1 \footnote{Their proof works with the minus version of the invariant but the same proof also works for the hat version.}). Also they have grid diagrams with grid number $7$. Hence, we construct $(p,q)$ cables for $\frac{q}{p}>15$ using Proposition \ref{construct}. Therefore, ${m(5_{2})}_{p,q}$ is a Legendrian non-simple link type for those values of $(p,q)$ by Theorem \ref{leg}. 
\end{proof}

\pagebreak

\end{document}